\documentclass[12pt]{amsart}

\usepackage{amsmath}
\usepackage{amsthm}
\usepackage{amssymb}
\usepackage{tikz}
\usepackage{mathtools}

\allowdisplaybreaks

\numberwithin{equation}{section}

\newcommand{\Z}{\mathbb{Z}}
\newcommand{\N}{\mathbb{N}}
\newcommand{\T}{\mathbb{T}}
\newcommand{\R}{\mathbb{R}}
\newcommand{\F}{\mathcal{F}}
\newcommand{\Sh}{\mathcal{S}}

\newcommand{\supp}{\mathop{\mathrm{supp}}}
\newcommand{\pa}{\partial}
\newcommand{\vphi}{\varphi}
\newcommand{\Op}{\mathop{\mathrm{Op}}}

\theoremstyle{plain}
\newtheorem{thm}{Theorem}[section]

\newtheorem{lem}[thm]{Lemma}

\newtheorem*{thmA}{Theorem A}

\theoremstyle{definition}

\begin{document}
\title[Bilinear pseudo-differential operators]
{A remark on the relationship $1/p = 1/p_1 + 1/p_2$ for boundedness of bilinear pseudo-differential operators with exotic symbols}

\author[T. Kato]{Tomoya Kato}
\author[N. Shida]{Naoto Shida}

\date{\today}

\address[T. Kato]
{Division of Pure and Applied Science, 
Faculty of Science and Technology, Gunma University, 
Kiryu, Gunma 376-8515, Japan}
\address[N. Shida]
{Department of Mathematics, 
Graduate School of Science, Osaka University, 
Toyonaka, Osaka 560-0043, Japan}

\email[T. Kato]{t.katou@gunma-u.ac.jp}
\email[N. Shida]{u331453f@ecs.osaka-u.ac.jp}

\keywords{Bilinear pseudo-differential operators, bilinear H\"ormander symbol classes}

\subjclass[2020]{35S05, 42B15, 42B35}

\thanks{This work was partially supported by JSPS KAKENHI Grant Numbers JP20K14339 (Kato).}

\begin{abstract}
We consider the bilinear pseudo-differential operators with symbols in the bilinear  
H\"ormander classes $BS_{\rho, \rho}^m$, $0 < \rho < 1$. 
In this paper, we show that
the condition $1/p = 1/p_1 + 1/p_2$ is necessary 
when we consider the boudnedness from $H^{p_1} \times H^{p_2}$ to $L^p$ 
of those operators for the critical case.
\end{abstract}

\maketitle

\section{Introduction} \label{Intro}
For $m \in \R$ and $0 \le \rho, \delta \le 1$, 
the bilinear H\"ormander symbol class $BS^{m}_{\rho, \delta}$ 
consists of all functions $\sigma(x, \xi_1, \xi_2) \in C^\infty(\R^n \times \R^n \times \R^n)$ 
such that
\begin{equation*}
 |\pa_x^\alpha \pa_{\xi_1}^{\beta_1} \pa_{\xi_2}^{\beta_2} \sigma(x, \xi_1, \xi_2)|
 \le
 C_{\alpha, \beta_1, \beta_2}
 (1 + |\xi_1| + |\xi_2|)^{m+\delta|\alpha|- \rho(|\beta_1| + |\beta_2|)}
\end{equation*}
for all multi-indices 
$\alpha, \beta_1, \beta_2 \in \N^n_0 = \{0, 1, 2, \dots \}^n$.
For a symbol $\sigma \in BS^m_{\rho, \delta}$, 
the bilinear pseudo-differential operator $T_\sigma$ is defined by
\begin{align*}
T_\sigma(f_1, f_2)(x)
=
\frac{1}{(2\pi)^{2n}}
\int_{(\R^n)^2}
e^{i x \cdot (\xi_1+ \xi_2)}
\sigma(x, \xi_1, \xi_2)
\widehat{f_1}(\xi_1)
\widehat{f_2}(\xi_2)
\,
d\xi_1 d\xi_2
\end{align*}
for $f_1, f_2 \in \Sh(\R^n)$.

Let $X_1$, $X_2$ and $Y$ be function spaces on $\R^n$
equipped with quasi-norms $\|\cdot\|_{X_1}$, $\|\cdot\|_{X_2}$ and $\|\cdot\|_Y$.
If there exists a constant $C>0$ such that the estimate
\begin{equation} \label{boundedness}
\|T_\sigma(f_1, f_2)\|_Y \le C \|f_1\|_{X_1}\|f_2\|_{X_2},
\quad
f_1 \in \Sh \cap X_1,
\quad
f_2 \in \Sh \cap X_2
\end{equation}
holds, then, with a slight abuse of terminology, 
we say that the operator $T_\sigma$ is bounded 
from $X_1 \times X_2$ to $Y$. 
The smallest constant $C$ of \eqref{boundedness} is denoted by $\|T_\sigma\|_{X_1 \times X_2 \to Y}$. 
$\Op(BS^m_{\rho, \delta})$ denotes  the class of all bilinear pseudo-differential operators $T_\sigma$ 
with $\sigma \in BS^m_{\rho, \delta}$. 
If all $T_\sigma$ with $\sigma \in BS^m_{\rho, \delta}$ are bounded from $X_1 \times X_2$ to $Y$, then we write
$\Op(BS^m_{\rho, \delta}) \subset B(X_1 \times X_2 \to Y)$.

The boundedness properties of the bilinear pseudo-differential operators with symbols 
in the bilinear H\"ormander classes have been studied in many researches.
In the case $\rho = 1$ and $\delta < 1$, 
bilinear pseudo-differential operators with symbols in $BS^{0}_{1, \delta}$ 
can be regarded as bilinear Calder\'on-Zygmund operators in the sense of Grafakos-Torres \cite{GT}, 
and they are bounded from $L^{p_1} \times L^{p_2}$ to $L^p$, 
$1 < p_1, p_2 < \infty$, $1/p = 1/p_1 + 1/p_2$. 
For details, see Coifman-Meyer \cite{CM}, B\'enyi-Torres \cite{BT-Symbolic} 
and B\'enyi-Maldonado-Naibo-Torres \cite{BMNT}.
We remark that the assumption 
$1/p = 1/p_1 + 1/p_2$
is reasonable since a symbol $\sigma \equiv 1$ is in $BS^0_{1, \delta}$ 
and the corresponding bilinear operator $T_\sigma$ is the product of two functions.

In this paper, we shall consider the case $0 \le \rho = \delta <1$.
In the linear case, the celebrated Calder\'on-Vaillancourt theorem \cite{CV} states 
that the linear pseudo-differential operator with the symbol satisfying the estimate 
\begin{equation*}
|\pa^\alpha_x \pa^\beta_\xi \sigma(x, \xi)|
\le
C_{\alpha, \beta} (1 + |\xi|)^{\rho|\alpha|-\rho|\beta|}
\end{equation*}
is bounded on $L^2$. 
However, in the bilinear settings,  
it was pointed out in B\'enyi-Torres \cite{BT-Almost} that
for any $1 \le p, p_1, p_2 < \infty$ satisfying $1/p = 1/p_1 + 1/p_2$, 
there exists a symbol $\sigma \in BS^0_{\rho, \rho}$ such that
$T_\sigma$ is not bounded from $L^{p_1} \times L^{p_2}$ to $L^p$.
Thus, there is certain difference between linear and bilinear cases.

For $0 \le \rho < 1$ and $0 < p_1, p_2 \le \infty$, we define
\begin{align*}
&m_\rho(p_1, p_2) = (1-\rho) m_0(p_1, p_2),
\\
&m_0(p_1, p_2)
=
-n 
\left(
\max
\left\{
\frac{1}{2}, 
\frac{1}{p_1},
\frac{1}{p_2},
1- \frac{1}{p_1} - \frac{1}{p_2},
\frac{1}{p_1}+ \frac{1}{p_2} -\frac{1}{2}
\right\}
\right).
\end{align*}
We see that 
the explicit values of $m_0(p_1, p_2)$ are 
\begin{align*}
m_0(p_1, p_2)
=
\begin{cases}
n/p_1 + n/p_2 - n \quad & \text{if}\ (1/p_1, 1/p_2) \in J_0;
\\
-n/2 \quad & \text{if}\ (1/p_1, 1/p_2) \in J_1;
\\
-n/p_2 \quad & \text{if}\ (1/p_1, 1/p_2) \in J_2;
\\
-n/p_1 \quad & \text{if}\ (1/p_1, 1/p_2) \in J_3;
\\
n/2 - n/p_1 - n/p_2 \quad & \text{if}\ (1/p_1, 1/p_2) \in J_4,
\end{cases}
\end{align*}
where $J_0, \dots, J_4$ are the following 5 regions;

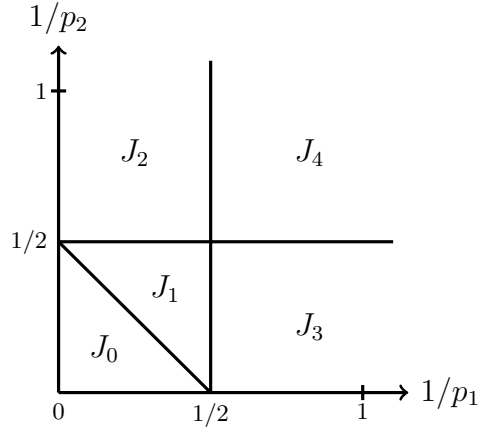
\begin{center}
\begin{tikzpicture}[scale=1.0]
\draw[very thick, black][->] (0, 0) -- (4.6, 0) node [right] {$1/p_1$};
\draw[very thick, black][->] (0, 0) -- (0, 4.6) node [above] {$1/p_2$};
\draw (0, 0)  node [below] {\scriptsize$0$};
\draw (2, 0)  node [below] {\scriptsize$1/2$};
\draw (0, 2)  node [left] {\scriptsize$1/2$};
\draw (4, 0)  node [below] {\scriptsize$1$};
\draw (0, 4)  node [left] {\scriptsize$1$};
\draw (0.6, 0.9)  node [below] {$J_0$};
\draw (1.4, 1.7)  node [below] {$J_1$};
\draw (1.0, 3.5)  node [below] {$J_2$};
\draw (3.3, 1.2)  node [below] {$J_3$};
\draw (3.3, 3.5)  node [below] {$J_4$};
\draw[very thick, black][-] (2, 0) -- (2, 4.4);
\draw[very thick, black][-] (0, 2) -- (4.4, 2);
\draw[very thick, black][-] (0, 2) -- (2, 0);
\draw[very thick, black][-] (-0.1, 4.0) -- (0.1, 4.0);
\draw[very thick, black][-] (4.0, -0.1) -- (4.0, 0.1);
\end{tikzpicture}
\end{center}
Then, Miyachi-Tomita \cite{MT-IUMJ, MT-I, MT-II} proved the following theorem.
\begin{thmA}
Let $0 \le \rho < 1$, $0 < p_1, p_2, p \le \infty$ and $1/p= 1/p_1 + 1/p_2$.
Then all bilinear pseudo-differential operators with symbols in $BS^{m_\rho(p_1, p_2)}_{\rho, \rho}$ 
are bounded from $H^{p_1} \times H^{p_2}$ to $L^p$, 
where $L^p$ should be replaced by $BMO$ if $p_1 = p_2 = p= \infty$.
\end{thmA}

\noindent
For the definition of the Hardy space $H^p$ and the space $BMO$, see Section \ref{preliminaries}.
In \cite{MT-IUMJ}, it was also proved that,
for $0 < p_1, p_2, p \le \infty$ with $1/p_1 + 1/p_2 = 1/p$, 
the class $BS^{m}_{\rho, \rho}$ with $m > m_{\rho}(p_1, p_2)$ does not assure 
the boundedness from $H^{p_1} \times H^{p_2}$ to $L^p$.
In this sense, the class $BS_{\rho, \rho}^{m_\rho(p_1, p_2)}$ can be regarded as a critical class.
We remark that the sharper boundedness was proved in \cite{MT-IUMJ} for the case $\rho = 0$ 
and in Kato \cite{Kato} for the case $0 \le \rho < 1$ and $(1/p_1, 1/p_2) \in J_1$, $1/p= 1/p_1 + 1/p_2$.
We also notice that the boundedness for the case $0 \le \rho < 1/2$ and $p_1 = p_2 = p= \infty$ 
was also obtained in Naibo \cite{Naibo}.
For the subcritical case $m < m_\rho(p_1, p_2)$, see 
B\'enyi-Bernicot-Maldonado-Naibo-Torres \cite{BBMNT} 
and 
Michalowski-Rule-Staubach \cite{MRS}. 

Now, we focus on the condition $1/p = 1/p_1 +1/p_2$, 
which is often assumed when the $L^{p_1} \times L^{p_2} \to L^p$ boundedness is considered.
It may be difficult to determine whether the assumption $1/p = 1/p_1 +1/p_2$ in Theorem A is valid or not
since the symbol $\sigma \equiv 1$ is not in $BS_{\rho, \rho}^{m_\rho(p_1, p_2)}$ 
for any $0\le \rho < 1$ and $0< p_1, p_2 \le \infty$.
Recently, Kato-Miyachi-Tomita \cite{KMT-bilinear, KMT-multilinear}  proved  
by using $L^2$-based amalgam spaces that 
if $(1/p_1, 1/p_2) \in J_1$, $1 \le p \le 2$ and $1/p \le 1/p_1 + 1/p_2$, 
then all bilinear operators $T_\sigma$ with $\sigma \in BS^{m_0(p_1, p_2)}_{0, 0} = BS^{-n/2}_{0, 0}$ 
are bounded in $L^{p_1} \times L^{p_2} \to L^p$.
From this result, we conclude that  the condition $1/p = 1/p_1 + 1/p_2$ is not always necessary
when we discuss the boundedness of bilinear pseudo-differential operators with symbols 
in the critical class of $\rho = 0$.

The purpose of this paper is to show that the condition $1/p = 1/p_1 + 1/p_2$
is essential to have the boundedness of bilinear pseudo-differential operators with 
symbols in $BS_{\rho, \rho}^{m_\rho(p_1, p_2)}$, $0< \rho < 1$.
The following is the main result of this paper.


\begin{thm} \label{mainthm}
Let $0 < \rho < 1$ and $0 < p_1, p_2, p \le \infty$. 
If 
all bilinear pseudo-differential operators with symbols in 
$BS^{m_{\rho}(p_1, p_2)}_{\rho, \rho}$ are bounded from 
$H^{p_1}(\R^n) \times H^{p_2}(\R^n)$ to $L^p(\R^n)$
when $p < \infty$  
or
to
$BMO(\R^n)$ when $p = \infty$, 
then $1/p = 1/p_1 + 1/p_2$.
\end{thm}


The contents of this paper are as follows.
In Section \ref{preliminaries}, we recall some basic notations and preliminary facts.
In Section \ref{mainsec}, we prove Theorem \ref{mainthm}.

\section{Preliminaries}\label{preliminaries}
For two nonnegative quantities $A$ and $B$,
the notation $A \lesssim B$ means that
$A \le CB$ for some unspecified constant $C>0$,
and $A \approx B$ means that
$A \lesssim B$ and $B \lesssim A$.
For $1 \le p \le \infty$,
$p'$ is the conjugate exponent of $p$,
that is, $1/p+1/p'=1$.


Let $\Sh(\R^n)$ and $\Sh'(\R^n)$ be the Schwartz space of
rapidly decreasing smooth functions on $\R^n$ and its dual,
the space of tempered distributions, respectively.
We define the Fourier transform $\F f$
and the inverse Fourier transform $\F^{-1}f$
of $f \in \Sh(\R^n)$ by
\[
\F f(\xi)
=\widehat{f}(\xi)
=\int_{\R^n}e^{-ix\cdot\xi} f(x)\, dx
\quad \text{and} \quad
\F^{-1}f(x)
=\frac{1}{(2\pi)^n}
\int_{\R^n}e^{ix\cdot \xi} f(\xi)\, d\xi.
\]


For a measurable subset $E \subset \R^n$, 
the Lebesgue space $L^p(E)$, $0< p \le \infty$, 
is the set of all those mesurable functions $f$ on $E$ 
such that 
$\|f\|_{L^p(E)} = (\int_E |f(x)|^p dx)^{1/p} < \infty$ if $0< p< \infty$ or 
$\|f\|_{L^\infty(E)} = \mathrm{ess} \sup_{x \in E} |f(x)| < \infty$ if $p = \infty$.
We also write $\|f\|_{L^p(E)} = \|f(x)\|_{L_x^p(E)}$
when we want to indicate the variable explicitly.


We recall the definition of Hardy spaces and the space $BMO$ on $\R^n$.
Let $\phi \in \Sh(\R^n)$ be such that $\int_{\R^n} \phi(x)\, dx =1$.
For $0 < p \le \infty$, the Hardy space $H^p(\R^n)$ consists of all $f \in \Sh^\prime(\R^n)$ such that
\begin{equation*}
\|f\|_{H^p}
=
\|
\sup_{t > 0}
|\phi_t * f|
\|_{L^p}
< 
\infty,
\end{equation*}
where $\phi_t(x) = t^{-n} \phi(t^{-1} x)$. 
It is known that
the definition of $H^p (\R^n)$ is independent of the choice of the function $\phi$,
and that
$H^p(\R^n) = L^p(\R^n)$ for $1 < p \le \infty$. 
The space $BMO(\R^n)$ consists of all locally integrable functions $f$ on $\R^n$
such that
\begin{equation*}
\|f\|_{BMO}
=
\sup_Q
\frac{1}{|Q|}
\int_{Q}
|f(x) -f_Q|
\, dx
< \infty, 
\end{equation*}
where the supremum is taken over all cubes in $\R^n$ 
and 
$f_Q = |Q|^{-1} \int_Q f$.
It is known that the dual space of $H^1(\R^n)$ is $BMO(\R^n)$.


\section{Proof of Theorem \ref{mainthm}}\label{mainsec}
In this section, we shall prove Theorem \ref{mainthm}.
For $0< \rho < 1$ and $1\le p_1, p_2 \le \infty$, we set 
\begin{align*}
\widetilde{m}_\rho(p_1, p_2)
&=
 (1-\rho)
\widetilde{m}_0(p_1, p_2),
\\
\widetilde{m}_0(p_1, p_2)
&=
-n
\left(
\max
\left\{
\frac{1}{2}, 
\frac{1}{p_1},
\frac{1}{p_2},
\frac{1}{p_1}+ \frac{1}{p_2} -\frac{1}{2}
\right\}
\right)
\\
&=
\begin{cases}
-n/2 \quad & \text{if}\ (1/p_1, 1/p_2) \in I_1;
\\
-n/p_2 \quad & \text{if}\ (1/p_1, 1/p_2) \in I_2;
\\
-n/p_1 \quad & \text{if}\ (1/p_1, 1/p_2) \in I_3;
\\
n/2 -n/p_1 -n/p_2 \quad & \text{if}\ (1/p_1, 1/p_2) \in I_4, 
\end{cases}
\end{align*} 
where $I_1, \dots, I_4$ are the following ranges;

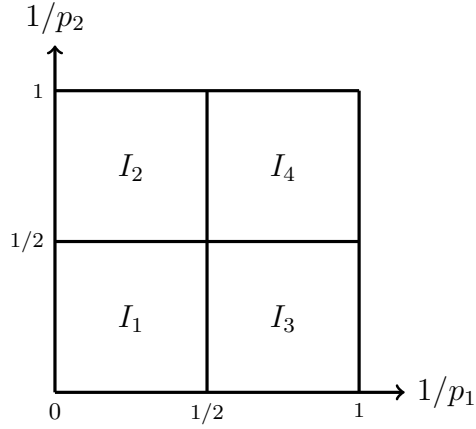
\begin{center}
\begin{tikzpicture}[scale=1.0]
\draw[very thick, black][->] (0, 0) -- (4.6, 0) node [right] {$1/p_1$};
\draw[very thick, black][->] (0, 0) -- (0, 4.6) node [above] {$1/p_2$};
\draw (0, 0)  node [below] {\scriptsize$0$};
\draw (2, 0)  node [below] {\tiny$1/2$};
\draw (0, 2)  node [left] {\tiny$1/2$};
\draw (4, 0)  node [below] {\tiny$1$};
\draw (0, 4)  node [left] {\tiny$1$};
\draw (1.0, 1.3)  node [below] {$I_1$};
\draw (1.0, 3.3)  node [below] {$I_2$};
\draw (3.0, 1.3)  node [below] {$I_3$};
\draw (3.0, 3.3)  node [below] {$I_4$};
\draw[very thick, black][-] (2, 0) -- (2, 4);
\draw[very thick, black][-] (0, 2) -- (4, 2);
\draw[very thick, black][-] (0, 4) -- (4, 4);
\draw[very thick, black][-] (4, 4) -- (4, 0);
\end{tikzpicture}
\end{center}

We now introduce the following fact proved in Wainger \cite[Theorem 10]{Wainger}. 
Let $0< a < 1$ and $0< b < n$.
For $t > 0$, we define
\begin{equation*}
\widetilde{f}_{a, b, t}(x)
=
\sum_{k \in \Z^n \setminus \{0\}}
e^{-t|k|} |k|^{-b} e^{i |k|^a} e^{i k \cdot x}.
\end{equation*}
Then, the limit $\widetilde{f}_{a, b}(x) = \lim_{t \to +0} \widetilde{f}_{a, b, t}(x)$ 
exists for all $x \in\R^n \setminus (2\pi\Z)^n$.
Furthermore, 
if $1 \le p \le \infty$ and $b > n/2 -an/2-n/p+an/p$, then $\widetilde{f}_{a, b} \in L^p(\T^n)$. 

From this, the following fact was given. 
\begin{lem}[{\cite[Lemma 6.1]{MT-IUMJ}}]\label{Wainger-example}
Let $ 0 < a < 1$, $0< b < n$ and $\vphi \in \Sh(\R^n)$.
For $t > 0$, we set
\begin{equation*}
f_{a, b, t}(x)
=
\sum_{k \in \Z^n \setminus \{0\}}
e^{-t|k|} |k|^{-b} e^{i |k|^a} e^{i k \cdot x} \vphi(x).
\end{equation*}
If $1 \le p \le \infty$ and 
$b > n - na/2 - n/p + na/p$,
then $\sup_{t > 0} \|f_{a, b, t}\|_{L^p(\R^n)} < \infty$.
\end{lem}

The following lemma plays an essential role in the proof of Theorem \ref{mainthm}

\begin{lem} \label{keylem}
Let $0 < \rho < 1$, $1 < p_1, p_2 < \infty$ and $ 0 < p < \infty$.
If all bilinear pseudo-differential operators with symbols in $BS^{\widetilde{m}_{\rho}(p_1, p_2)}_{\rho, \rho}$
 are bounded from $L^{p_1}(\R^n) \times L^{p_2}(\R^n)$ to $L^p(\R^n)$,
then $1/p = 1/p_1 + 1/p_2$.
\end{lem}

\begin{proof}
In this proof, we simply write
$\widetilde{m}_\rho(p_1, p_2)$ 
and
$\widetilde{m}_0(p_1, p_2)$
as
$m_\rho$ and $m_0$, respectively.

We first show that $1/p \le 1/p_1 + 1/p_2$.
If the symbol $\sigma$ is independent of $x$, namely, $\sigma =\sigma(\xi_1, \xi_2)$, 
then $T_\sigma$ is called a bilinear Fourier multiplier operator. 
We will use the following fact proved in \cite[Proposition 7.3.7]{Grafakos-Modern}:
If a nonzero bilinear Fourier multiplier operator is bounded from $L^{p_1} \times L^{p_2}$ to $L^p$, 
$0 < p_1, p_2, p < \infty$, then $1/p \le 1/p_1 + 1/p_2$.
Now, let $\sigma = \sigma(\xi_1, \xi_2) \in \Sh(\R^{2n})$.
Then, since $\sigma \in  BS^{\widetilde{m}}_{\rho, \rho}$
for any $\widetilde{m} \le 0$,
it follows from our assumption and the above fact that
$1/p \le 1/p_1 + 1/p_2$.

Next, we prove $1/p \ge 1/p_1 + 1/p_2$. 
The proof below is based on the idea given by \cite[Proof of Lemma 6.3]{MT-IUMJ}. 
We first give a rough argument omitting necessary limiting argument 
and then we incorporate necessary details at the last part of the proof.
By our assumption and the closed graph theorem, there exists a positive integer $M$
such that
\begin{multline*}
\|T_\sigma\|_{L^{p_1} \times L^{p_2} \to L^{p}}
\\
\lesssim
\max_{|\alpha|, |\beta_1|, |\beta_2| \le M}
\|
(1+|\xi_1|+|\xi_2|)^{-m_\rho-\rho(|\alpha| -|\beta_1|-|\beta_2|)}
\pa_x^{\alpha}
\pa_{\xi_1}^{\beta_1}
\pa_{\xi_2}^{\beta_2}
\sigma(x, \xi_1, \xi_2)
\|_{L^\infty}.
\end{multline*}
for all $\sigma \in BS^{m_\rho}_{\rho, \rho}$ (see \cite[Lemma 2.6]{BBMNT}). 

Let $\varphi, \widetilde{\varphi} \in \Sh(\R^n)$ be such that
\begin{align*}
&\supp \varphi \subset [-1/4, 1/4]^n,
\quad
|\F^{-1}\varphi(x)| \ge 1 \ \text{on} \ [-1, 1]^n,
\\
&\supp \widetilde{\varphi} \subset [-1/2, 1/2]^n,
\quad
\widetilde{\varphi} = 1 \ \text{on} \ [-1/4, 1/4]^n.
\end{align*}
For $\epsilon > 0$ and $0< a_1, a_2 < 1$, we set
\begin{align*}
b_i = n -\frac{na_i}{2} -\frac{n}{p_i} + \frac{na_i}{p_i} + \epsilon,
\quad
i = 1, 2.
\end{align*}
For sufficiently large $j \in \N$, 
we define 
\begin{align*}
&\sigma(\xi_1, \xi_2)
=
\sum_{(k_1, k_2) \in D_j}
c_{k_1, k_2} (1+ |k_1| + |k_2|)^{m_0} 
\varphi(2^{-j\rho}\xi_1 -k_1)
\varphi(2^{-j\rho}\xi_2 -k_2),
\\
&\widehat{f_{a_k, b_k}}(\xi_k)
=
\sum_{\ell_k \in \Z^n \setminus \{0\}} 
|\ell_k|^{-b_k} e^{i |\ell_k|^{a_k}} 
\widetilde{\varphi}(2^{-j\rho}\xi_k - \ell_k),
\quad
k = 1, 2, 
\end{align*}
where
$\{c_{k_1, k_2}\}$ is a sequence satisfying $\sup_{k_1, k_2 \in \Z^n} |c_{k_1, k_2}| \le 1$,
and 
\begin{equation*}
D_j
=
\left\{
(k_1, k_2) \in \Z^n \times \Z^n
\ :\ 
|k_1| \approx |k_2| \approx |k_1+ k_2| \approx 2^{j(1-\rho)} 
\right\}.
\end{equation*}
Then, since $|\xi_1|, |\xi_2| \approx 2^j$ for $(\xi_1, \xi_2) \in \supp \sigma$,
we have
\begin{equation*}
|\pa_{\xi_1}^{\beta_1}\pa_{\xi_2}^{\beta_2} \sigma(\xi_1, \xi_2)|
\lesssim
2^{j(1-\rho)m_0} 2^{-j\rho(|\beta_1| +|\beta_2|)}
\approx
(1 + |\xi_1| + |\xi_2|)^{m_\rho -\rho(|\beta_1| + |\beta_2|)},
\end{equation*}
and thus, this $\sigma$ belongs to $BS_{\rho, \rho}^{m_\rho}$.
Here, it should be emphasized that the implicit constants are independent of $\{c_{k_1, k_2}\}$. 
It follows from Lemma \ref{Wainger-example} that 
\begin{align*}
\|f_{a_1, b_1}\|_{L^{p_1}}
&=
\Big\|
\sum_{\ell_1 \in \Z^n \setminus \{0\}} 
|\ell_1|^{-b_1} e^{i |\ell_1|^{a_1}} e^{i 2^{j\rho} x \cdot \ell_1}
2^{j\rho n} \widetilde{\Phi}(2^{j\rho} x)
\Big\|_{L^{p_1}_x}
\\
&=
2^{j\rho n(1 - 1/p_1)}
\Big\|
\sum_{\ell_1 \in \Z^n \setminus \{0\}} 
|\ell_1|^{-b_1} e^{i |\ell_1|^{a_1}} e^{i x \cdot \ell_1}
\widetilde{\Phi}(x)
\Big\|_{L^{p_1}_x}
\approx
2^{j\rho n(1 - 1/p_1)},
\end{align*}
where $\widetilde{\Phi} = \F^{-1} \widetilde{\varphi}$. 
To have the equivalence above, we need a slight modification for the definition of $f_{a_1, b_1}$ 
(see the last part of this proof). 
Similarly, $\|f_{a_2, b_2}\|_{L^{p_2}} \approx 2^{j \rho n (1 -1/p_2)}$.

On the other hand, 
since $\varphi(\cdot -k) \widetilde{\varphi}(\cdot -\ell) = \varphi(\cdot -k)$ if $k = \ell$, and 0 otherwise,  
we have
\begin{align*}
T_\sigma(f_{a_1, b_1}, f_{a_2, b_2})(x)
&=
\sum_{(k_1, k_2) \in D_j}
c_{k_1, k_2}
(1+ |k_1| + |k_2|)^{m_0}
|k_1|^{-b_1} |k_2|^{-b_2}
\\
&\qquad\qquad\qquad\qquad
\times
e^{i|k_1|^{a_1}}
e^{i|k_2|^{a_2}}
e^{i 2^{j\rho}x \cdot (k_1+k_2)}
\{
2^{j\rho n}
\Phi(2^{j \rho} x)
\}^2
\end{align*}
with $\Phi = \F^{-1}\varphi$.
Now, let $\{r_k(\omega)\}_{k \in \Z^n}$ be a sequence of the Rademacher functions on $\omega \in [0, 1]^n$. 
We choose $\{c_{k_1, k_2}\}$ as  
$c_{k_1, k_2} = r_{k_1+k_2}(\omega) e^{-i|k_1|^{a_1}}e^{-i|k_2|^{a_2}}$. 
Then, 
\begin{align*}
&T_\sigma(f_{a_1, b_1}, f_{a_2, b_2})(x)
\\
&=
\sum_{(k_1, k_2) \in D_j}
r_{k_1+k_2}(\omega)
(1+ |k_1| + |k_2|)^{m_0}
|k_1|^{-b_1} |k_2|^{-b_2}
e^{i 2^{j\rho}x \cdot (k_1+k_2)}
\{
2^{j\rho n}
\Phi(2^{j \rho} x)
\}^2
\\
&=
\{
2^{j\rho n}
\Phi(2^{j \rho} x)
\}^2
\sum_{|k| \approx 2^{j(1-\rho)}} 
r_k(\omega) e^{i 2^{j\rho}x \cdot k} d_k,
\end{align*}
where 
\begin{align*}
d_k
=
\sum_{k_1 : |k_1| \approx |k-k_1| \approx 2^{j(1-\rho)}}
(1+ |k_1| + |k-k_1|)^{m_0}
|k_1|^{-b_1} |k-k_1|^{-b_2}.
\end{align*}
Here, we see from the way to generate $d_k$ that the cardinality of the set 
$\{k_1 \in \Z^n : |k_1| \approx |k-k_1| \approx 2^{j(1-\rho)}\}$
for $|k| \approx 2^{j(1-\rho)}$ is equivalent to $2^{j(1-\rho) n}$.
By a change of variables 
and 
the assumption $|\Phi| \ge 1$ on $[-1, 1]^n$, 
we have
\begin{align*}
	\|T_\sigma(f_{a_1, b_1}, f_{a_2, b_2})\|_{L^p}
&=
2^{j\rho n(2 -1/p)}
\Big\|
\{
\Phi(x)
\}^2
\sum_{|k| \approx 2^{j(1-\rho)}} 
r_k(\omega) e^{ix \cdot k} d_k
\Big\|_{L^p_x}
\\
&\ge
2^{j\rho n(2 -1/p)}
\Big\|
\sum_{|k| \approx 2^{j(1-\rho)}} 
r_k(\omega) e^{ix \cdot k} d_k
\Big\|_{L^p_x ([-1, 1]^n)}.
\end{align*}
Thus, our assumption implies that
\begin{align} \label{Estimate-are}
2^{j \rho n (1/p -1/p_1-1/p_2)}
\gtrsim
\Big\|
\sum_{|k| \approx 2^{j(1-\rho)}} 
r_k(\omega) e^{ix \cdot k} d_k
\Big\|_{L^p_x([-1, 1]^n)}.
\end{align}
It should be emphasized that the implicit constant in this inequality
does not depend on $\omega \in [0, 1]^n$.
Raising \eqref{Estimate-are} to a $p$th-power and integrating over $\omega$, 
we obtain
\begin{align*}
(2^{j \rho n (1/p -1/p_1-1/p_2)})^p
\gtrsim
\int_{[-1, 1]^n}
\int_{[0, 1]^n}
\Big|
\sum_{|k| \approx 2^{j(1-\rho)}} 
r_k(\omega) e^{ix \cdot k} d_k
\Big|^p
d\omega
dx.
\end{align*}
By Khintchine's inequality (see, e.g., \cite[Appendix C]{Grafakos-Classical}), the right hand side is equivalent to
\begin{align*}
\int_{[-1, 1]^n}
\Big(
\sum_{|k| \approx 2^{j(1-\rho)}} 
|e^{ix \cdot k} d_k|^2
\Big)^{p/2}
dx
\approx
\Big(
\sum_{|k| \approx 2^{j(1-\rho)}} 
|d_k|^2
\Big)^{p/2}.
\end{align*}
Since 
\begin{align*}
|d_k|
&\approx
2^{j(1-\rho)(m_0-b_1-b_2)}
\Big(
\sum_{k_1 : |k_1| \approx |k-k_1| \approx 2^{j(1-\rho)}}
1
\Big)
\approx
2^{j(1-\rho)(m_0-b_1-b_2 + n)}
\end{align*}
for $|k| \approx 2^{j(1-\rho)}$, we obtain
\begin{align*}
\Big(
\sum_{|k| \approx 2^{j(1-\rho)}} 
|d_k|^2
\Big)^{p/2}
&\approx
2^{j(1-\rho)(m_0-b_1-b_2 + n) p}
\Big(
\sum_{|k| \approx 2^{j(1-\rho)}} 
1
\Big)^{p/2}
\approx
2^{j(1-\rho)(m_0-b_1-b_2 + 3n/2)p}.
\end{align*}
Combining the above estimates, we have
\begin{align*}
2^{j \rho n (1/p -1/p_1-1/p_2)}
\gtrsim
2^{j(1-\rho)(m_0-b_1-b_2 + 3n/2)}, 
\end{align*}
and since $j$ is arbitrarily large, this implies that
\begin{align} \label{Estimate-exponent}
\begin{split}
&\rho n 
\left(
\frac{1}{p} -\frac{1}{p_1} - \frac{1}{p_2}
\right)
\\
&\ge
(1-\rho)
\left\{
m_0
+
\left(
\frac{n}{p_1} + \frac{n}{p_2} - \frac{n}{2}
\right)
+
a_1
\left(
\frac{n}{2} - \frac{n}{p_1}
\right)
+
a_2
\left(
\frac{n}{2} - \frac{n}{p_2}
\right)
- 2\epsilon
\right\}.
\end{split}
\end{align}
As mentioned at the beginning of this section, we have
\begin{align*}
m_0
=
\begin{cases}
-n/2 \quad & \text{if}\ (1/p_1, 1/p_2) \in I_1;
\\
-n/p_2 \quad & \text{if}\ (1/p_1, 1/p_2) \in I_2;
\\
-n/p_1 \quad & \text{if}\ (1/p_1, 1/p_2) \in I_3;
\\
n/2 -n/p_1 - n/p_2  \quad & \text{if}\ (1/p_1, 1/p_2) \in I_4,
\end{cases}
\end{align*}
and hence, we see that 
the right hand side of \eqref{Estimate-exponent} converges to 0 as $\epsilon \to 0$ and 
\begin{align*}
\begin{cases}
a_1 \to 1, a_2 \to 1 \quad & \text{if}\ (1/p_1, 1/p_2) \in I_1;
\\
a_1 \to 1, a_2 \to 0 \quad & \text{if}\ (1/p_1, 1/p_2) \in I_2;
\\
a_1 \to 0, a_2 \to 1 \quad & \text{if}\ (1/p_1, 1/p_2) \in I_3;
\\
a_1 \to 0, a_2 \to 0 \quad & \text{if}\ (1/p_1, 1/p_2) \in I_4.
\end{cases}
\end{align*}
Thus, since $\rho > 0$, we obtain $1/p- 1/p_1 -1/p_2 \ge 0$, 
that is, $1/p \ge 1/p_1 + 1/p_2$.

We notice that the above argument is not rigorous 
since $f_{a_1, b_1}$ and $f_{a_2, b_2}$ are not in $\Sh$.
In order to correct this point, 
we replace $f_{a_k, b_k}$ by $f_{a_k, b_k, t}$ satisfying
\begin{align*}
\widehat{f_{a_k, b_k, t}}(\xi_k)
=
\sum_{\ell_k \in \Z^n \setminus \{0\}}
e^{-t|\ell_k|} |\ell_k|^{-b_k} e^{i |\ell_k|^{a_k}}  
\widetilde{\vphi}(2^{-j\rho}\xi_k - \ell_k),
\quad
k = 1, 2,
\end{align*}
where $a_k$ and $b_k$ are the same as above.
Then, $f_{a_k, b_k, t} \in \Sh$, 
and Lemma \ref{Wainger-example} implies that 
$\sup_{t > 0} \|f_{a_k, b_k, t}\|_{L^{p_k}} \approx 2^{j \rho n (1-1/p_k)}$, 
$k= 1, 2$. 
In the same way as above, we have
\begin{align*}
T_{\sigma}(f_{a_1, b_1, t}, f_{a_2, b_2, t})(x)
=
\{
2^{j\rho n}
\Phi(2^{j \rho} x)
\}^2
\sum_{|k| \approx 2^{j(1-\rho)}} 
r_k(\omega) e^{i 2^{j\rho}x \cdot k} d_{k, t}
\end{align*}
with
\begin{align*}
d_{k, t}
=
\sum_{k_1 : |k_1| \approx |k-k_1| \approx 2^{j(1-\rho)}}
(1+ |k_1| + |k-k_1|)^{m_0}
e^{-t(|k_1| + |k - k_1|)}
|k_1|^{-b_1} |k-k_1|^{-b_2}.
\end{align*}
Since this is the sum of a finite number of terms, 
we have $d_{k, t} \to d_k$ as $t \to 0$.
The rest of the proof is the same as above.
\end{proof}


In the rest of this section,  
we will complete the proof of Theorem \ref{mainthm}.
Let $J_0, \dots, J_4$ be the same region as in Section \ref{Intro}. 
The following proof is also based on the idea by
\cite[Theorem 6.4]{MT-IUMJ}.

First, we assume $p<\infty$.
We consider the case $(1/p_1, 1/p_2) \in (0, 1)^2$.
Since $H^p = L^p$, $1 < p \le \infty$,
the desired results for 
$(1/p_1, 1/p_2) \in \bigcup_{k=1}^4 J_k \cap (0, 1)^2$ 
follow from Lemma \ref{keylem}.
In the case $(1/p_1, 1/p_2) \in J_0 \cap (0, 1)^2$, 
we will use a duality argument. 
It should be remarked that  
$p$ is restricted to $2 \le p < \infty$ in this case,
since $1/p \le 1/p_1 + 1/p_2 \le 1/2$ for $(1/p_1, 1/p_2) \in J_0 \cap (0, 1)^2$
(for the first inequality, 
see the argument in the proof of Lemma \ref{keylem}).  
It is known that
$T_{\sigma}$ is bounded from $L^{p_1} \times L^{p_2}$ to $L^p$
if and only if
$T_{\sigma^{\ast1}}$ is bounded from $L^{p'} \times L^{p_2}$ to $L^{p_1'}$, 
where $\sigma^{\ast1}$ is given by
\[
\int_{\R^n} T_{\sigma} (f_1,f_2)(x) g(x) \,dx 
= \int_{\R^n} T_{\sigma^{\ast1}} (g,f_2)(x) f_1(x) \,dx .
\]
It is further known that if $\sigma$ is in $BS_{\rho,\rho}^{m}$,
then the transposition $\sigma^{\ast1}$ is also in $BS_{\rho,\rho}^{m}$
(see \cite[Theorem 2.1]{BMNT}).
Thus, it follows from these facts that
if all the operators in $\Op(BS_{\rho,\rho}^{m})$ are
bounded from $L^{p_1} \times L^{p_2}$ to $L^p$,
then all the operators in $\Op(BS_{\rho,\rho}^{m})$ are
bounded from $L^{p'} \times L^{p_2}$ to $L^{p_1'}$.
Now, since we have $BS_{\rho, \rho}^{-(1-\rho)n/p^\prime}
\subset BS_{\rho, \rho}^{-(1-\rho)n(1-1/p_1-1/p_2)}$
by recalling $1/p \le 1/p_1 + 1/p_2$,
our assumption and the above argument yield that
all $T_\sigma$ with $\sigma \in BS_{\rho, \rho}^{-(1-\rho)n/p^\prime}$ are 
bounded from $L^{p^\prime} \times L^{p_2}$ to $L^{p_1^\prime}$.
Since $(1/p^\prime, 1/p_2) \in J_3 \cap (0, 1)^2$ and $1 < p_1' < 2$, by Lemma  \ref{keylem}, 
we have $1/p_1^\prime = 1/p^\prime + 1/p_2$, that is, 
$1/p = 1/p_1 + 1/p_2$.

For the case $(1/p_1, 1/p_2) \in [0, \infty)^2 \setminus (0, 1)^2$,
we first consider the case 
$(1/p_1, 1/p_2) \in J_2 \setminus (0, 1)^2$.
The same argument works for the other cases 
$J_k \cap (0, 1)^2$, $k=1,3,4$.
Suppose that all bilinear pseudo-differential operators 
with symbols in $BS^{-(1-\rho)n/p_2}_{\rho, \rho}$ are bounded from $H^{p_1} \times H^{p_2}$ to $L^p$.
Since 
\begin{equation}\label{L2L2L1}
\Op(BS^{-(1-\rho)n/2}_{\rho, \rho}) 
\subset
B( L^2 \times L^2 \to L^1 )
\end{equation}
(see \cite{MT-I}),  
by interpolation, 
all $T_\sigma$ with $\sigma \in BS^{-(1-\rho)n/\widetilde{p_2}}_{\rho, \rho}$ 
are bounded from $H^{\widetilde{p}_1} \times H^{\widetilde{p}_2}$ to $L^{\widetilde{p}}$ 
with $(1/\widetilde{p}_1, 1/\widetilde{p}_2, 1/\widetilde{p})
=
(1-\theta) (1/2, 1/2, 1) + \theta (1/p_1, 1/p_2, 1/p)$ 
and $0< \theta < 1$. 
Taking $\theta$ sufficiently close to 0, 
we have $(1/\widetilde{p}_1, 1/\widetilde{p}_2) \in J_2 \cap (0, 1)^2$. 
By Lemma \ref{keylem}, 
we have
\begin{equation*}
0 
=
\frac{1}{\widetilde{p}} - \frac{1}{\widetilde{p}_1} -\frac{1}{\widetilde{p}_2}
=
\theta
\left(
\frac{1}{p} - \frac{1}{p_1} -\frac{1}{p_2}
\right),
\end{equation*}
and this implies $1/p  = 1/p_1 + 1/p_2$.
The remaining part is $(1/p_1, 1/p_2) \in J_0 \setminus (0, 1)^2$.
The case $J_0 \setminus \{ (0, 1)^2 \cup \{(0,0)\} \}$ is similar to the above.
In fact, using interpolation with
\begin{equation}\label{L1LbtoL2}
\Op(BS^{-(1-\rho)n/2}_{\rho, \rho}) 
\subset B(L^a \times L^b \to L^2),
\quad
1/a,\, 1/b > 0, 
\quad
1/a + 1/b = 1/2,
\end{equation}
which was also proved in \cite{MT-I}, 
and repeating the same argument just above, 
we obtain the desired result
from the fact for $J_0 \cap (0, 1)^2$ proved in the previous step.
Finally, we consider the case $(1/p_1, 1/p_2) = (0, 0)$.
However, this choice is invalid.
In fact, toward a contradiction, 
we assume $(1/p_1, 1/p_2) = (0, 0)$,
that is,
the boundedness
$\Op(BS^{-(1-\rho)n}_{\rho, \rho}) 
\subset B(L^\infty \times L^\infty \to L^p)$
holds.
Interpolating this with \eqref{L1LbtoL2},
we have 
\[
\Op(BS_{\rho, \rho}^{ -(1-\rho)n(1 - 1/\widetilde{p}_1 - 1/\widetilde{p}_2)} ) 
\subset
B(L^{\widetilde{p}_1} \times L^{\widetilde{p}_2} \to L^{\widetilde{p}}),
\]
where $(1/\widetilde{p}_1, 1/\widetilde{p}_2, 1/\widetilde{p})
=
(1-\theta) (1/a, 1/b, 1/2) + \theta (0, 0, 1/p)$
and $0<\theta<1$.
Then from the fact for $J_0 \cap (0, 1)^2$,
$1/\widetilde{p} = 1/{\widetilde{p}_1} + 1/{\widetilde{p}_2}$ holds,
which is identical with $\theta/p=0$.
However, this is impossible since we are now assuming $p<\infty$.
Therefore, we complete the proof for the case $p<\infty$.

Next, we consider the case $p=\infty$, namely,
our assumption is 
\begin{equation}\label{assuforcont}
\Op(BS^{m_{\rho}(p_1, p_2)}_{\rho, \rho}) \subset
B \big( H^{p_1} \times H^{p_2} \to BMO \big).
\end{equation}
For the case $p_1=p_2=\infty$ is clear, since $1/\infty + 1/\infty = 1/\infty$.
In what follows, we assume $(1/p_1, 1/p_2) \in [0,\infty)^2 \setminus \{(0,0)\}$.
We first consider the case $(1/p_1, 1/p_2) \in [0,1)^2 \setminus \{(0,0)\}$.
However, this choice is invalid. 
In fact, we have
$\Op(BS^{m_{\rho}(p_1, p_2)}_{\rho, \rho}) \subset
B (H^1 \times L^{p_2} \to L^{p_1'} )$
by duality of \eqref{assuforcont}.
Since $1 \le p_1' < \infty$, 
the conclusion for $p<\infty$ gives
$1+1/p_2=1/p_1'$, which is identical with $1/p_1+1/p_2=0$.
However, this is impossible.
Next, we consider the case $(1/p_1, 1/p_2) \in [0,\infty)^2 \setminus [0,1)^2$.
However, this case is also invalid. 
In fact, by interpolating between
\eqref{assuforcont} in this case and \eqref{L2L2L1}, 
we have
\[
\Op(BS_{\rho, \rho}^{ m_{\rho}(\widetilde{p}_1,\widetilde{p}_2)} ) \subset
B(L^{\widetilde{p}_1} \times L^{\widetilde{p}_2} \to L^{\widetilde{p}}),
\]
where
$(1/\widetilde{p}_1, 1/\widetilde{p}_2, 1/\widetilde{p})
=
(1-\theta) (1/2, 1/2, 1) + \theta (1/p_1, 1/p_2, 0)$ 
and $0 < \theta < 1$.
Since $1/\widetilde{p} = 1-\theta$, namely, 
$1 < \widetilde{p} <\infty$, we have 
$1/\widetilde{p} = 1/{\widetilde{p}_1} + 1/{\widetilde{p}_2}$
by the conclusion for $p<\infty$.
This coincides with $\theta(1/p_1+1/p_2)=0$.
However, this is again impossible.
Consequently, in the case $p=\infty$,
the possible choice is only $p_1=p_2=\infty$,
which means $1/p = 1/p_1 + 1/p_2$.

The proof of Theorem \ref{mainthm} 
is complete.



\section*{Acknowledgement}
The authors sincerely express their thanks to Professor Akihiko Miyachi and
Professor Naohito Tomita for valuable discussions and fruitful advices.


\end{document}